\documentclass[10pt, reqno]{amsart}

\usepackage[english]{babel}
\usepackage[latin1]{inputenc}
\usepackage{babel}
\usepackage{fancyhdr}
\usepackage{amsmath,amsfonts,amssymb,amsthm}
\usepackage[mathcal]{euscript}
\usepackage{mathrsfs}
\usepackage[tmargin=1.5in,bmargin=1.5in,lmargin=1.3in,rmargin=1.3in]{geometry}
\usepackage[all]{xy}
\usepackage{textcomp}
\usepackage{epsfig}
\usepackage{graphics}
\usepackage{graphicx}
\usepackage{mathrsfs}

\usepackage{caption}
\usepackage[OT2,OT1]{fontenc}
\newtheorem{thm}{Theorem}[section]

\newtheorem{prop}[thm]{Proposition}

\usepackage{microtype}
\usepackage{bbm}
\usepackage{color}
\usepackage[textsize=small]{todonotes}       % includes TO DO LIST

\definecolor{halfgray}{gray}{0.55}
\definecolor{webgreen}{rgb}{0,.5,0}
\definecolor{webbrown}{rgb}{.6,0,0}
\definecolor{Maroon}{cmyk}{0, 0.87, 0.68, 0.32}
\definecolor{royalblue}{cmyk}{1, 0.50, 0, 0}
\definecolor{Black}{cmyk}{0, 0, 0, 0}

\usepackage{hyperref}
\hypersetup{%
    colorlinks=true, linktocpage=true, pdfstartpage=1, pdfstartview=FitV,%
    breaklinks=true, pdfpagemode=UseNone, pageanchor=true, pdfpagemode=UseOutlines,%
    plainpages=false, bookmarksnumbered, bookmarksopen=true, bookmarksopenlevel=1,%
    hypertexnames=true, pdfhighlight=/O,%nesting=true,%frenchlinks,%
    urlcolor=webbrown, linkcolor=royalblue, citecolor=webgreen, %pagecolor=RoyalBlue,%
}
\numberwithin{equation}{section}

\newcommand{\Var}{\mathrm{Var}}
\renewcommand{\P}{\mathbb{P}}

\newcommand{\indic}{\mathbbm{1}}

\newcommand{\db}{\overline d}

\newcommand{\E}{\mathbb{E}}
\newcommand{\e}{\mathrm{e}}

\renewcommand{\P}{\mathbb{P}}

\DeclareMathOperator*{\essinf}{ess\,inf}
\newcommand{\ER}{Erd\H{o}s-R\'enyi  \ }

\newcommand{\eqn}[1]{\begin{equation}#1\end{equation}}
\newcommand{\eqan}[1]{\begin{align}#1\end{align}}
\newcommand{\nn}{\nonumber}

\begin{document}

\title{Connectivity of Poissonian Inhomogeneous random Multigraphs}
\author{Lorenzo Federico}

\begin{abstract}

We introduce a new way to sample inhomogeneous random graphs designed to have a lot of flexibility in the assignment of the degree sequence and the individual edge probabilities while remaining tractable. To achieve this we run a Poisson point process over the square $[0,1]^2$, with an intensity proportional to a kernel $W(x,y)$ and identify every couple of vertices of the graph with a subset of the square, adding an edge between them if there is a point in such subset.  This ensures unconditional independence among edges and makes many statements much easier to prove in this setting than in other similar models. Here we prove sharpness of the connectivity threshold under mild integrability conditions on $W(x,y)$.
\medskip


\end{abstract}
\maketitle

\section{Introduction and model description}\label{sect:model}

In this paper we introduce a new model to generate inhomogeneous random multigraphs on $n$ vertices in which edges are sampled independently according to two parameters:
\begin{itemize}
\item A sequence $(t_n)_{n \geq 2}$ that controls the expected total number of edges in the multigraph.
\item A symmetric kernel  $W(x,y) : [0,1]^2 \to \mathbb R_{\geq 0}$ that indicates which edges have higher probability to be present.
\end{itemize}

%\section{The model}

%We now describe a way to sample a sequence of inhomogeneous random multigraphs $G_n(W,t)$ on $n$ vertices  starting from a symmetric kernel $W(x,y) : [0,1]^2 \to \mathbb R_{\geq 0}$ and a sequence of positive real numbers $(t_n)_{n\geq 2}$.

For every $n\geq 2$ we define the vertex set of the graph $G_n(W,t)$ as $V_n:=\{v_i;i\in [n]\}$ and for every $i \in [n]$ we define the interval $S_i:=((i-1)/n,i/n]$. We run a Poisson point process over the square $[0,1]$ with intensity $t_n W(x,y)$ and add an edge between $ v_i$ and $v_j$, $i\leq j$, for every point in the square $S_i\times S_j$.

This is equivalent to adding between any couple of vertices $\{v_i, v_j\}$, $i\leq j$, a number of edges distributed as a Poisson random variable, whose parameter $\lambda_{ij}$ is given by

\begin{equation}
\lambda_{ij}:=\int_{S_i\times S_j}t_n W(x,y)dxdy,
\end{equation}
independent of each other.
We also define the random graph $\tilde G_n(W,t)$ obtained from the multigraph $G_n(W,t)$ by erasing the multiedges and self-loops. In $\tilde G_n(W,t)$ every edge $\{v_i,v_j\}$ is present with probability
\begin{equation}\label{eq:probedge}
p_{ij}:=1- \exp\Big\{ -\int_{S_i\times S_j}t_n W(x,y)dxdy \Big\},
\end{equation}
independent of the others. It is straightforward to see that $\tilde G_n(W,t)$ is connected if and only if $G_n(W,t)$ is connected. If $W(x,y) \in \mathbb L_1([0,1]^2)$, then $t_n \lVert W\rVert _1/2$ is the expected number of edges in $G_n(W,t)$, so tuning opportunely the sequence $(t_n)_{n \geq 2}$ we can use this procedure to generate multigraphs with any given density of edges. 
For a given constant $c$, taking $t_n=cn^2$ results in a dense graph, while taking $t_n=cn$ we obtain a sparse graph with finite average degree. Note that $\tilde G_n(W,t)$ in some special cases is asymptotically equivalent to many well-known models, such as the \ER random graph \cite{ErdRen59}, if the kernel $W(x,y)$ is constant, the Norros-Reittu random graph \cite{NorRei06}, if  $W(x,y)=f(x)f(y)$ for some function $f:[0,1]\to \mathbb R_{\geq 0}$, and the stochastic block model \cite{HolLasLei83}, if the kernel is piecewise constant. Note however that there are also several popular models that cannot be expressed in terms of a sequence $\tilde G_n(W,t)$ for any kernel $W$, such as percolation on sparse graphs or random intersection graphs (see \cite{ErdSpe79,FilSchSin00, Joos13,  Ryb11} for some connectivity results about those models).

This model is also closely related to the general inhomogeneous random graph model defined by Bollob\'as, Janson and Riordan in  \cite{BolJanRio07} which is also defined by a kernel $W(x,y)$, but is sampled by placing $n$ points at random iid positions $(x_i)_{i\in [n]}$ on the interval $[0,1]$ and then adding the edge $\{v_i,v_j\}$ to the graph with a probability that is given by some function of $W(x_i,x_j)$. In such setting the connectivity threshold was computed in \cite{DevFra14}, finding similar results to those we present here, under stricter conditions on the kernel. The main advantage of our definition of the graph is that, since we identify the vertices with the deterministic intervals $(S_i)_{ i \in [n]}$ instead of the random positions $x_i$, the edges are present independently of each other unconditionally and not only given the positions $(x_i)_{i \in [n]}$. We see in the proof of the main theorem of this paper how this property makes many arguments much easier and sometimes allows for a completely different approach.

The fact that we are sampling our graphs from a kernel $W(x,y):[0,1]^2\mapsto \mathbb R_+$ suggests that this model might converge to some graphon, in the sense described in \cite{BorChaLovSosVes08}, in the dense regime (i.e. when $t_n=cn^2$). This is the case, with the limit graphon given by $1-\e^{-cW(x,y)}$, as indicated by \eqref{eq:probedge}.

\section{The main theorem}

In this section we formulate  the main theorem of this paper, about the connectivity threshold of the inhomogeneous multigraph we described, and discuss the conditions required to prove it.

We take $t_n=c n\log n$ for some $c \geq 0$.
We define

\begin{equation}
H(x):=\int_{[0,1]} W(x,y)dy,
\quad
\nu_{0}:=\essinf_{[0,1]} H(x).
\end{equation}
%Moreover we define for a given $q$
%
%\begin{equation}
%\nu_q(x):=\Big( \int_{[0,1]} W(x,y)^q dy\Big)^{1/q}
%\end{equation}
%and assume that there exists a $q >1$ such that $\nu_q (x)\in \mathbb L_\infty([0,1])$.
We also require the kernel $W(x,y)$ to be irreducible, which means that there is no set $A\subset [0,1]$ such that $0<\mu(A)<1$ and $\int_{A \times A^c}W(x,y)dxdy=0$.
What follows is the main theorem of the present paper:

\begin{thm}\label{thm:main}
Consider a sequence of graphs $G_n(W,t)$, with $t_n= c n \log n$ and $W(x,y)$ irreducible.

%$\nu_q (x)\in \mathbb L_\infty([0,1])$
If $c<1/\nu_0$ and $H(x) \in \mathbb L_1^{loc}(F)$ for some open set $F\subseteq [0,1]$  such that $\mu(F)=1$, then 

\eqn{
\lim_{n \to \infty} \P(G_n(W,t) \text{ is connected})=0.
}

If $c>1/\nu_0$ and $W(x,y) \in \mathbb L_q([0,1]^2)$ for some $q>2$, then

\begin{equation}
\lim_{n \to \infty} \P(G_n(W,t) \text{ is connected})= 1.
\end{equation}
\end{thm}

In other words, under some integrability conditions, the graph is connected whp if there are no vertices with an expected degree lower than $\log n$ and if there are not two sets of vertices which are deterministically separated.
We divide the proof in several steps. First we analyze the threshold for the existence of isolated vertices and prove that it coincides with what we claim to be the connectivity threshold. Then, we prove that when $c > \nu_0$ the graph is actually connected, providing two separate arguments for the non existence whp of small and large components.

Note that for the upper bound to hold we require the condition $W(x,y)\in \mathbb L_q([0,1]^2)$, which might seem counterintuitive since in most connectivity proofs (see \cite{DevFra14,ErdRen59,FedHof16}) the most important role is played by the vertices of low degree, while the vertices of high degree are almost irrelevant, since they tend to be always part of the giant component. This condition is necessary because a vertex might have a high expected degree just because it is given a very large number of self loops or multiple edges, which do not actually contribute to connectivity. The $\mathbb L_q$ condition is required to ensure that this effect is not too drastic. It is easy to see, by stochastic domination, that if a kernel $W(x,y)\notin \mathbb L_q$ can be lower bounded by another kernel $W'(x,y)\in \mathbb L_q$ such that $G_n(W',t)$ satisfies the conditions we require for it to be connected whp, then also $G_n(W,t)$ is connected whp.

In this paper we do not discuss what happens if $c=1/\nu_0$ or more in general if $t_n/(n \log n) \to 1/\nu_0$, because in that regime the asymptotic probability of connectivity of the graph behaves differently based on the specific shape of the kernel $W$ and it is hard to give general formulas stated in term of relatively easy and natural conditions.

\section{Connection Probabilities}\label{sect:connp}
We first give a simple formula for the probability that a given set of vertices in $G_n(W,t)$ has no outgoing edges. 
For every $A\subset [n]$ we define the set $B_A \subset [0,1]$ as $B_A:=\bigcup_{v_i \in A} S_i$,

and the event $C_A$ as the event that all the edges between $A$ and $A^c$ are vacant, i.e., that $A$ is the union of connected components. This is a crucial notion for the present paper, as the graph $G_n(W,t)$ is connected if and only if there is no proper subset $A$ of $[n]$ such that the event $C_A$ happens. 
Thus, we need a compact formula for the probability of $C_A$. Define the set $[0,1]^2_x:=\{(x,y) \in [0,1]^2: x<y\}$. By the definition of $G_n(W,t)$  we  write

%\eqan{\label{eq:PCA}
%\P(C_A)&=\exp\Big\{- t_n \int_{B_A \times B_A^c}W(x,y)dxdy\Big\}\\&=\exp\Big\{- t_n \int_{B_A }\Big(H(x)-\int_{B_A} W(x,y)dy\Big)dx\Big\}.\nn
%}

\eqan{\label{eq:PCA}
\P(C_A)&=\exp\Big\{- t_n \int_{(B_A \times B_A^c\cup B_A^c\times B_A)\cap [0,1]^2_x}W(x,y)dxdy\Big\}\nn
\\&=\exp\Big\{- t_n \int_{B_A \times B_A^c}W(x,y)dxdy\Big\}
\\&=\exp\Big\{- t_n \int_{B_A }H(x)dx-\int_{B_A\times B_A} W(x,y)dydx\Big\},\nn
}
where in the second equality we used the symmetry of $W(x,y)$. 

\section{The lower bound: the number of isolated vertices}

As in most connectivity proofs, the relevant parameter for connectivity of the graph is the number $Y_n$ of isolated vertices. We first prove that the limit behavior of $Y_n$ is mainly determined by its expectation in a much more general setting, requiring only that edges are sampled independently, without assuming that the edge probabilities are defined using $W(x,y)$ and $t_n$. Then we compute bounds on $\E[Y_n]$ in the specific case of $G_n(W,t)$.

\subsection{Law of large number for isolated vertices}
We first define a more general inhomogeneous random graph in which edges are sampled independently but we do not ask for any regularity on the edge probabilities $p_{ij}$.
Given a number $n$ and an array $\mathbf P=(p_{ij})_{i<j\leq n}  \in [0,1]^{{[n] \choose 2}}$, we define the random graph $G_n( \mathbf P)$ with vertex set
$\{v_i, i \in [n]\}$, in which each edge $e_{ij}:=\{v_i, v_j\}$ is present with probability $p_{ij}$ independent of the others.

We will prove that the existence of isolated points in $G_n( \mathbf P)$ is regulated mainly by the first moment of their number. This result can be deduced from the main theorem of \cite{FalLarMar16}  with some effort, seeing the edge addition process as a coupon collector over the vertices, but we provide here a short and direct proof to improve readability of the paper.

The graph  $\tilde G_n(W,t)$ is a special case of $G_n(\mathbf P)$ in which the probabilities $p_{ij}$ are defined by \eqref{eq:probedge}, moreover, every vertex is isolated in $\tilde G_n(W,t)$ if and only if it is isolated in $G_n(W,t)$ (for our purposes we consider vertices with only self loops as isolated), so the following theorem can be applied to both models.
Define $Y_n$ as the number of isolated vertices in $G_n(\mathbf P)$. We prove the following result about the concentration of the number of isolated points:
\begin{thm}\label{llniso}
Consider a sequence of random graphs $G_n(\mathbf P)$ such that $\E[Y_n]/\log n \to \infty$,
then 
\eqn{
\Var(Y_n)/\E[Y_n]^2 \to 0.
}
\proof

We write, defining the event $I_i=\{v_i$ is isolated$\}$ for every $i \in [n]$,

\eqan{
\Var(Y_n) &= \E[Y_n^2]-\E[Y_n]^2= \sum_{i,j}\P(I_i \cap I_j)-\sum_{i,j} \P(I_i)\P(I_j)\nn\\
&=\sum_{i,j}\P(I_i)\P(I_j\mid I_i)-\sum_{i,j} \P(I_i)\P(I_j)= \sum_i \P(I_i)\sum_j \big(\P(I_j\mid I_i)-\P(I_j)\big)
}
We take care of the elements of the sum such that $i=j$ with the following upper bound:
\eqan{
 \sum_i \P(I_i)\big(\P(I_i\mid I_i)-\P(I_i)\big)= \sum_i \P(I_i)(1-\P(I_i))\leq\sum_i \P(I_i)= \E[Y_n]
}
If $i\neq j$ we note that $\P(I_j\mid I_i)=\P(I_j)/(1-p_{ij})$, so that we can rewrite

\eqn{
\P(I_j\mid I_i)-\P(I_j)=\dfrac{\P(I_j)}{1-p_{ij}}-\P(I_j)=\P(I_j)\dfrac{p_{ij}}{1-p_{ij}}.
}
We thus obtain

\eqan{
\Var(Y_n)\leq \E[Y_n]+\sum_{i}  \P(I_i)\sum_{j\neq i} \P(I_j)\dfrac{p_{ij}}{1-p_{ij}}.
}
Define the expected degree of the vertex $v_i$ as $\db_i=\sum_{j\neq i}p_{ij}$. To take care of the vertices $v_i$ such that $\db_i \geq 3 \log n$, we obtain

\eqan{
 \sum_{i: \db_i\geq 3 \log n}  \P(I_i)\sum_{j\neq i} \P(I_j)\dfrac{p_{ij}}{1-p_{ij}}&\leq \sum_{i: \db_i\geq 3 \log n}  \Big(1-\dfrac{\db_i}{n-1}\Big)^{n-1}\sum_{j\neq i} \P(I_j)\dfrac{p_{ij}}{1-p_{ij}}\nn\\
&\leq \e^{-3 \log n(1+o(1))}n^2 \to 0,\label{eq:isodeg}
}
using that for every $i,j$ 

\eqn{\label{eq:condbound}
 \P(I_j)\dfrac{p_{ij}}{1-p_{ij}}=p_{ij}\prod_{h\neq i,j}(1-p_{jh})\leq 1.
}
To control the vertices such that $\db_i< 3 \log n$ instead, we write, for an arbitrary $\varepsilon >0$,

\eqan{\label{eq:isoleq}
\sum_{i: \db_i< 3 \log n}  \P(I_i)\sum_{j\neq i} \P(I_j)\dfrac{p_{ij}}{1-p_{ij}}=\sum_{i: \db_i< 3 \log n}  \P(I_i)\Big(\sum_{j: p_{ij}\leq \varepsilon} \P(I_j)\dfrac{p_{ij}}{1-p_{ij}}+\sum_{j: p_{ij}> \varepsilon} \P(I_j)\dfrac{p_{ij}}{1-p_{ij}}\Big).
}
We again bound

\eqn{
\sum_{i: \db_i< 3 \log n}  \P(I_i)\sum_{j: p_{ij}\leq \varepsilon} \P(I_j)\dfrac{p_{ij}}{1-p_{ij}}\leq \sum_{i}  \P(I_i)\sum_j \P(I_j) \dfrac{\varepsilon}{1-\varepsilon}=\E[Y_n]^2\dfrac{\varepsilon}{1-\varepsilon}.
}
On the other hand, if $\db_i \leq 3 \log n$, then there are at most $(3/\varepsilon)\log n$ distinct $j$s such that $p_{ij}> \varepsilon$, so, using again  \eqref{eq:condbound},
\eqan{\label{eq:isovep}
\sum_{i: \db_i< 3 \log n}  \P(I_i)\sum_{j: p_{ij}> \varepsilon} \P(I_j)\dfrac{p_{ij}}{1-p_{ij}}&\leq \sum_{i: \db_i< 3 \log n}  \P(I_i)\sum_{j} \indic_{\{p_{ij}> \varepsilon\}}\leq \E[Y_n] (3/\varepsilon)\log n.
}
Consequently, summing \eqref{eq:isodeg}, \eqref{eq:isoleq} and \eqref{eq:isovep}, we obtain that for every $\varepsilon > 0$
\eqn{
\Var(Y_n) \leq \E[Y_n]+\E[Y_n]^2\dfrac{\varepsilon}{1-\varepsilon}+ \E[Y_n](3/\varepsilon)\log n,
}
so that, since we assumed $\E[Y_n]/\log n \to \infty$,
\eqn{
\limsup_{n \to \infty}\Var(Y_n)/\E[Y_n]^2 \leq \dfrac{\varepsilon}{1-\varepsilon},
}
from the fact that $\varepsilon$ is arbitrary the claim follows. \qed
\end{thm}
\subsection{The expected number of isolated vertices} 
We now study the asymptotic of $\E[Y_n]$ for $G_n(W,t)$, to prove that the threshold for the existence of isolated vertices is indeed the claimed threshold for connectivity in Theorem \ref{thm:main}. 
% We compute
%
%\begin{equation}
%\P(I_i) =\exp\Big\{ - \int_{[(i-1)/n,i/n]}t_nH(x)dx \Big\},
%\end{equation}
%so that
%\begin{equation}
%\max_i \P(I_i)\leq \exp\Big\{ -t_n \nu_0/n \Big\}=n^{-c\nu_0},
%\end{equation}
%consequently  if $c \geq 1/\nu_0$, then
%
%\begin{equation}
%\E[Y_n]\leq n \max_i \P(I_i) \leq n^{1-c\nu_0} \to 0.
%\end{equation}
%We conclude by first moment method the proof that when $c \geq 1/\nu_0$ whp there are no isolated vertices in $G_n(W,t)$.

We prove that when $c< 1/\nu_0$, $\E[Y_n]\gg \log n$, so that we can apply Theorem \ref{llniso}. If $c < 1/\nu_0$, then for some $\varepsilon >0$, there exists a set $A\subseteq [0,1]$ such that 

\begin{equation}\mu(A)>\varepsilon; \quad \sup_{x \in A}t_nH(x)< (1-\varepsilon ) n\log n.
\end{equation}

We next define the sequence of functions $H_n(x)$ as

\begin{equation}
H_n(x)=n\int_{[\lfloor xn\rfloor/n,\lceil xn \rceil/n]}H(x)dx.
\end{equation}
Note that $H_n(x)$ is constant over the intervals $((i-1)/n, i/n)$ and is not properly defined for $x=i/n$ for some $i$. To solve this issue we extend $H_n(x)$ so that it is left continuous. Since for every $n$, $\mu(\{i/n; i\in [n]\})=0$, this choice does not impact any of the following arguments. 
We recall that, for every vertex $v_i$,  $I_i=C(\{i\})$. By \eqref{eq:probedge}, for every $x \in S_i$, recalling \eqref{eq:PCA},

\eqan{
\P(I_i)\geq \exp\Big\{-t_n \int_{S_i}H(x)dx\Big\}   \geq \e^{-t_nH_n(x)/n}.\label{eq:hnprob}
}
We assumed the existence of an open set $F \subseteq [0,1]$ such that $\mu(F)=1$ and $H(x) \in \mathbb L_1^{loc}(F)$. By Lebesgue's differentiation theorem, we have that $H_n(x) \to H(x)$ almost everywhere in $F$ and thus almost everywhere in $[0,1]$. Consequently, by Egorov's theorem, there exist a set $B$ such that $\mu(B)< \varepsilon /2$ and a number $m$ such that, for every $n>m$,

\begin{equation}
\sup_{[0,1]\setminus B}| H_n(x)-H(x)|< \varepsilon/2.
\end{equation}
Consequently,

\begin{equation}
\mu(A\setminus B)\geq \varepsilon /2; \quad \sup_{A\setminus B}t_nH_n(x)< (1-\varepsilon/2) n\log n.
\end{equation}

We define the set $M_n(\varepsilon):=\{x: t_nH_n(x)< (1-\varepsilon/2) n\log n \}$. We know that for every $n>m$, $\mu(M_n(\varepsilon))> \varepsilon /2$, and that $M_n(\varepsilon)$ is the disjoint union of intervals of the form $((i-1)/n, i/n]$. We write 
\eqn{
V_n(\varepsilon):=\{ i \in [n]: ((i-1)/n, i/n]\subseteq M_n(\varepsilon)\},
}
for every $n>m$. Using \eqref{eq:hnprob}, we obtain

\eqn{
|V_n(\varepsilon)|>n \varepsilon/2; \qquad \min_{i \in V_n(\varepsilon)} \P(I_i) \geq n^{1-\varepsilon/2}.
}
Thus, we conclude
\eqn{
\E[Y_n] \geq \sum_{i \in V_n(\varepsilon)} \P(I_i) \geq  \frac{n\varepsilon}{2} n^{1-\varepsilon/2}\gg \log n,
}
and consequently, using Theorem \ref{llniso} we obtain, by Chebyshev's inequality,

\eqn{\label{eq:noiso}
\P(Y_n=0)=\P(Y_n \leq 0) \leq \dfrac{\Var(Y_n)}{\E[Y_n]^2} \to 0.
}

\section{The Upper bound}
In this section we prove the upper bound on the connectivity threshold, that is, that when $c>1/\nu_0$ the graph is connected whp. The proof is divided in two steps, first we show that whp there are no small components and then that there are not multiple giant components.
\subsection{No small components}\label{sect:nosmall}

We next prove that if $c>1/\nu_0$, then exists an $\varepsilon>0$ such that whp every component has size at least $n\varepsilon$.

We can now prove that whp all the components of $G_n(W,t)$ are large:
\begin{prop}\label{prop:nosmall}
Consider a sequence of graphs $G_n(W,t)$, with an irreducible kernel $W(x,y)\in \mathbb L_q([0,1]^2)$ for some $q>2$ and $t_n= c n \log n$ with $c \nu_0>1$. Then there exists $\varepsilon >0$ such that 
\eqn{
\P\Big(\bigcup_{ A : |A|<\varepsilon n} C_A\Big) \to 0.
}
\end{prop}
\proof

We prove the claim using the union bound, that is, computing that 
\eqn{
\sum_{A:|A|\leq \varepsilon n}\P(C_A) \to 0.
}
We next bound for all $A$, recalling \eqref{eq:PCA} and using the H\"older's inequality,

\begin{equation}
\int_{B_A\times B_A} W(x,y) dydx\leq \mu(B_A)^{2-2/q}\Big(\int_{B_A\times B_A} W(x,y)^q dydx\Big)^{1/q},
\end{equation}
 for any $q>1$, so that, for every $B_A$ such that $\mu(B_A) \leq \varepsilon$,
\begin{equation}
\int_{B_A\times B_A} W(x,y) dydx\leq  \varepsilon^{2-2/q}\lVert W\rVert_q.
\end{equation}

So, choosing $\varepsilon$, $q$ such that $\varepsilon^{2-2/q}\lVert W\rVert_q< \varepsilon\frac{\nu_0-1/c}{2}$, which is possible because of the assumptions we made on $W(x,y)$ in Section \ref{sect:model}, we obtain

\eqan{
\P(C_A)&\leq\exp\Big\{-t_n \frac{|A|}{n}\Big(\nu_0-\frac{\nu_0-1/c}{2}\Big)\Big\}\leq \exp\Big\{ -t_n \int_{B_A }\frac{\nu_0+ 1/c}{2}dx\Big\}
\\&\leq \exp\Big\{-c n\log n \frac{|A|}{n} \frac{\nu_0+1/c}{2}\Big\} 
=\exp\Big\{- |A|\log n\frac{c \nu_0+1}{2}\Big\}=n^{-|A|(1+\delta)},
}
for some appropriate $\delta>0$.
Finally
\eqn{
\sum_{A: |A|<\varepsilon n}\P(C_A) \leq \sum_{i=1}^{\varepsilon n} {n \choose i}n^{-i(1+\delta)}\leq \sum_{i=1}^{\varepsilon n}n^{-i \delta} \to 0.
}
\qed

\subsection{No multiple giants}

Next, we prove that for every $\varepsilon >0$, there cannot be a set of vertices of size at least $\varepsilon n$ that is not connected to its complementary.

\begin{prop}\label{prop:nolarge}
Consider a sequence of graphs $G_n(W,t)$, with an irreducible kernel $W(x,y)\in \mathbb L_q([0,1]^2)$ for some $q>2$ and $t_n= c n \log n$ with $c \nu_0>1$. Then for every $\varepsilon >0$, 

\eqn{
\P\Big(\bigcup_{ A : \varepsilon n<|A|<n/2} C_A\Big) \to 0.
}
\end{prop}
\proof

Recall the definitions of $B_A$ and $C_A$ given at the beginning of Section \ref{sect:connp}. Even when $|A|> \varepsilon n$, the first equality in \eqref{eq:PCA} applies. By definition $\mu (B_A) = |A|/n$ .
By \cite[Lemma 7]{BolBorCha10}\footnote{ The result is originally proved for bounded kernels, but if \eqref{eq:mincut} holds for the kernel $W'(x,y):=\max\{W(x,y),1\}$ it holds also for $W(x,y)$ by domination, and $W'(x,y)$ is irreducible if and only if $W(x,y)$ is irreducible.}, we know that for every $\varepsilon >0$, if $W(x,y)$ is irreducible,

\eqn{\label{eq:mincut}
\inf_{B: \varepsilon \leq \mu(B)\leq 1/2)}\int_{B \times B^c}W(x,y)dxdy=\delta (W, \varepsilon)>0,
}
so that, by \eqref{eq:PCA}

\eqn{
\max_{A: \varepsilon n \leq |A|\leq n/2)}\P(C_A)\leq \sup_{B: \varepsilon \leq \mu(B)\leq 1/2)}\exp\Big\{-t_n\int_{B \times B^c}W(x,y)dxdy\Big\}\leq \e^{-t_n\delta (W, \varepsilon)}.
}
Thus, we bound using again the first moment method

\eqan{
\P\Big(\bigcup_{ A : \varepsilon n<|A|<n/2} C_A\Big)&\leq \sum_{A:\varepsilon n<|A|\leq n/2}\P(C_A) \leq 2^n \e^{-t_n\delta (W, \varepsilon)}\\&=\e^{-n (c \delta (W, \varepsilon) \log n-\log 2)}\to 0. \nn\qed
}

We can finally use all the results we obtained to prove Theorem \ref{thm:main}.

\begin{proof}[Proof of Theorem \ref{thm:main}]

We know that

\eqn{
\P(G_n(W,t) \text{ is connected}) \leq \P(Y_n=0),
}
so by \eqref{eq:noiso} it follows that if $c \leq 1/\nu_0$, then $\P(G_n(W,t) \text{ is connected})\to 0$.

On the other hand, for $G_n(W,t)$ to be disconnected, there must exist a set $A$ of at most $n/2$ vertices such that $C_A$ happens. By Propositions \ref{prop:nosmall} ad \ref{prop:nolarge}, ve obtain that for $c>1/\nu_0$,

\eqn{
\P\Big(\bigcup_{ A : |A|\leq n/2} C_A\Big)\leq \P\Big(\bigcup_{ A : |A|<\varepsilon n} C_A\Big) +\P\Big(\bigcup_{ A : \varepsilon n<|A|\leq n/2} C_A\Big) \to 0,
}
so the claim follows.
\end{proof}
	\section*{Acknowledgments}

The work in this paper is supported by the European Research Council (ERC) through Starting Grant Random Graph, Geometry and Convergence 639046. The author would like to thank Agelos Georgakopoulos for introducing him to the model and Christoforos Panagiotis for the help with some analytic details of the proofs.
\begin{small}
\bibliographystyle{abbrv}
\bibliography{LorenzosBib}

\begin{thebibliography}{10}

\bibitem{BolBorCha10}
B.~Bollob\'{a}s, C.~Borgs, J.~Chayes, and O.~Riordan.
\newblock Percolation on dense graph sequences.
\newblock {\em Ann. Probab.}, 38(1):150--183, (2010).

\bibitem{BolJanRio07}
B.~Bollob\'as, S.~Janson, and O.~Riordan.
\newblock The phase transition in inhomogeneous random graphs.
\newblock {\em Random Structures Algorithms}, 31(1):3--122, (2007).

\bibitem{BorChaLovSosVes08}
C.~Borgs, J.~T. Chayes, L.~Lov\'{a}sz, V.~T. S\'{o}s, and K.~Vesztergombi.
\newblock Convergent sequences of dense graphs. {I}. {S}ubgraph frequencies,
  metric properties and testing.
\newblock {\em Adv. Math.}, 219(6):1801--1851, (2008).

\bibitem{DevFra14}
L.~Devroye and N.~Fraiman.
\newblock Connectivity of inhomogeneous random graphs.
\newblock {\em Random Structures Algorithms}, 45(3):408--420, (2014).

\bibitem{ErdRen59}
P.~Erd{\H{o}}s and A.~R{\'e}nyi.
\newblock On random graphs. {I}.
\newblock {\em Publ. Math. Debrecen}, {\bf 6}:290--297, (1959).

\bibitem{ErdSpe79}
P.~Erd{\H{o}}s and J.~Spencer.
\newblock Evolution of the {$n$}-cube.
\newblock {\em Comput. Math. Appl.}, {\bf 5}(1):33--39, (1979).

\bibitem{FalLarMar16}
V.~Falgas-Ravry, J.~Larsson, and K.~Markstr\"om.
\newblock Speed and concentration of the covering time for structured coupon
  collectors.
\newblock {\it arXiv:1601.04455 }, (2016).

\bibitem{FedHof16}
L.~Federico and R.~v.~d. Hofstad.
\newblock Critical window for connectivity in the configuration model.
\newblock {\it arXiv:1603.03254}, (2016).

\bibitem{FilSchSin00}
J.~A. Fill, E.~R. Scheinerman, and K.~B. Singer-Cohen.
\newblock Random intersection graphs when {$m=\omega(n)$}: an equivalence
  theorem relating the evolution of the {$G(n,m,p)$} and {$G(n,p)$} models.
\newblock {\em Random Structures Algorithms}, 16(2):156--176, (2000).

\bibitem{HolLasLei83}
P.~W. Holland, K.~B. Laskey, S.~Leinhardt, and and.
\newblock Stochastic blockmodels: first steps.
\newblock {\em Social Networks}, 5(2):109--137, (1983).

\bibitem{Joos13}
F.~Joos.
\newblock Random subgraphs in sparse graphs.
\newblock {\it arXiv:1312.0732}, (2013).

\bibitem{NorRei06}
I.~Norros and H.~Reittu.
\newblock On a conditionally {P}oissonian graph process.
\newblock {\em Adv. in Appl. Probab.}, 38(1):59--75, (2006).

\bibitem{Ryb11}
K.~Rybarczyk.
\newblock Diameter, connectivity, and phase transition of the uniform random
  intersection graph.
\newblock {\em Discrete Math.}, 311(17):1998--2019, (2011).

\end{thebibliography}
\end{small}

\end{document}